\documentclass[11pt,a4paper]{amsart}
\usepackage{amssymb}
\usepackage{amscd}

\theoremstyle{plain}
\newtheorem{thm}{Theorem}[section]
\newtheorem{lem}[thm]{Lemma}
\newtheorem{pro}[thm]{Proposition}

\theoremstyle{definition}
\newtheorem{dfn}[thm]{Definition}

\newcommand{\Z}{\mathbb{Z}}
\newcommand{\N}{\mathbb{N}}

\DeclareMathOperator{\res}{res}
\DeclareMathOperator{\Hom}{Hom}
\DeclareMathOperator{\Ima}{Im}

\DeclareMathOperator{\End}{End}
\DeclareMathOperator{\Ext}{Ext}

\newcommand{\W}{\mathcal{W}}
\DeclareMathOperator*{\rightoverright}{\parbox{2em}{\centerline{$\longrightarrow$}\vskip -6pt\centerline{$\longrightarrow$}}}
\DeclareMathOperator*{\bullop}{\parbox{1em}{\centerline{$\bullet$}}}

\begin{document}
\title[Infinite-dimensional string algebras]{Classification of modules for infinite-dimensional string algebras}

\author{William Crawley-Boevey}
\address{Department of Pure Mathematics, University of Leeds, Leeds LS2 9JT, UK}
\email{w.crawley-boevey@leeds.ac.uk}




\thanks{This material is based upon work supported by the National Science Foundation under Grant No. 0932078 000,
while the author was in residence at the Mathematical Science Research Institute (MSRI) in Berkeley, California,
during the spring semester 2013.}

\begin{abstract}
We relax the definition of a string algebra to also include infinite-dimensional algebras such as $k[x,y]/(xy)$.
Using the functorial filtration method, which goes back to Gelfand and Ponomarev,
we show that finitely generated and artinian modules (and more generally finitely controlled and pointwise artinian modules)
are classified in terms of string and band modules.
This subsumes the known classifications of finite-dimensional modules
for string algebras and of finitely generated modules for $k[x,y]/(xy)$.
Unlike in the finite-dimensional case, the words parameterizing string modules may be infinite.
\end{abstract}
\maketitle

\section{Introduction}
By a \emph{string algebra} we mean an algebra of the form $\Lambda=kQ/(\rho)$ 
where $k$ is a field, $Q$ is a quiver, not necessarily finite, $kQ$ is the path algebra,
$\rho$ is a set of zero relations in $kQ$, that is, paths of length~$\ge 2$,
$(\rho)$ denotes the ideal generated by $\rho$, and we suppose that
\begin{itemize}
\item[(a)]Any vertex of $Q$ is the head of at most two arrows and the tail of at most two arrows, and
\item[(b)]Given any arrow $y$ in $Q$, there is at most one path $xy$ of length~2 with $xy\notin \rho$
and at most one path $yz$ of length~2 with $yz\notin \rho$.
\end{itemize}
The name is due to Butler and Ringel~\cite{BR}, but they imposed a finiteness 
condition (which we drop), which forces an algebra with a one to be finite-dimensional. 
Note that the notion has a longer history, 
going back to the special biserial algebras of Skowro\'{n}ski and Waschb\"{u}sch \cite{SW}.

We consider left $\Lambda$-modules $M$ which are
\emph{unital} in the sense that $\Lambda M=M$.
If $Q$ is finite, then $\Lambda$ has a one, and this corresponds to the usual notion.
It is equivalent that $M$ is the direct sum of its subspaces $e_v M$, where $v$ runs through the
vertices in $Q$ and $e_v$ denotes the
trivial path at vertex $v$, considered as an idempotent element in $\Lambda$.
This ensures that $\Lambda$-modules correspond to representations
of $Q$ satisfying the zero relations in $\rho$,
with the vector space at vertex $v$ being $e_v M$.

As usual, a module $M$ is \emph{finitely generated} if
$M = \Lambda m_1+\dots + \Lambda m_n$
for some elements $m_1,\dots,m_n\in M$.
Slightly more generally, if $Q$ has infinitely many vertices,
we say that a module $M$ is \emph{finitely controlled} if for every vertex $v$, the set
$e_v M$ is contained in a finitely generated submodule of~$M$.
Similarly, slightly more general than the notion of an artinian module, we say that a module $M$ is
\emph{pointwise artinian} if for any descending chain of submodules $M_1 \supseteq M_2 \supseteq M_3 \supseteq \dots$
and any vertex $v$ in $Q$, the chain of subspaces $e_v M_1 \supseteq e_v M_2 \supseteq e_v M_3 \supseteq \dots$
stabilizes.

Given a string algebra, our main results classify modules satisfying these finiteness conditions in terms 
of so-called `string' and `band' modules. The results apply in particular
to the string algebra $k[x,y]/(xy)$, which arises from the quiver with one vertex and loops $x$ and $y$ with $\rho = \{xy,yx\}$.
Also to the algebra $k\langle x,y\rangle/(x^2,y^2)$, with $\rho = \{x^2,y^2\}$.
As another example, one can take $\Gamma = kQ/(\rho)$ where $Q$ is the quiver

\medskip

\[
\cdots
\rightoverright^{x_{-1}}_{y_{-1}}
\bullop^{-1}
\rightoverright^{x_0}_{y_0}
\bullop^0
\rightoverright^{x_1}_{y_1}
\bullop^1
\rightoverright^{x_2}_{y_2}
\bullop^2
\rightoverright^{x_3}_{y_3}
\cdots
\]

\medskip
\noindent
and $\rho = \{ x_i y_{i-1} : i\in\Z\} \cup \{ y_i x_{i-1} : i\in\Z\}$. 
Clearly $\Gamma$-modules 
are the same thing as $\Z$-graded modules for $k[x,y]/(xy)$, where $x$ and $y$ have degree 1,
and finitely controlled $\Gamma$-modules
correspond to $\Z$-graded $k[x,y]/(xy)$-modules whose homogeneous components are finite-dimensional.

\subsection*{Words}
As in previous work on string algebras, in order to describe the string and band modules, we use certain `words', and as in~\cite{Ri2}, they may be infinite.
These words are also used to to define functors used in the proofs, and it is for this purpose that there are two trivial words for each vertex.
By a \emph{letter} $\ell$ one means either an arrow $x$ in $Q$ (a \emph{direct} letter)
or its formal inverse $x^{-1}$ (an \emph{inverse} letter).
The head and tail of an arrow $x$ are already defined,
and we extend them to all letters so that
the head of $x^{-1}$ is the tail of $x$ and vice versa.
If $I$ is one of the sets
$\{0,1,\dots,n\}$ with $n\ge 0$, or $\N = \{0,1,2,\dots\}$, or $-\N = \{0,-1,-2,\dots\}$ or $\Z$,
we define an \emph{$I$-word} $C$ as follows.
If $I\neq \{0\}$, then $C$ consists of a sequence of letters $C_i$ for all $i\in I$ with $i-1\in I$, so
\[
C = \begin{cases}
C_1 C_2 \dots C_n & \text{(if $I = \{0,1,\dots,n\}$)}
\\
C_1 C_2 C_3 \dots & \text{(if $I = \N$)}
\\
\dots C_{-2} C_{-1} C_0 & \text{(if $I = -\N$)}
\\
\dots C_{-1} C_0 \vert C_1 C_2 \dots & \text{(if $I = \Z$)}
\end{cases}
\]
(a bar shows the position of $C_0$ and $C_1$ if $I=\Z$) satisfying:
\begin{itemize}
\item[(a)]if $C_i$ and $C_{i+1}$ are consecutive letters,
then the tail of $C_i$ is equal to the head of $C_{i+1}$;
\item[(b)]if $C_i$ and $C_{i+1}$ are consecutive letters, then $C_i^{-1}\neq C_{i+1}$; and
\item[(c)]no zero relation $x_1\dots x_m$ in $\rho$, nor its
inverse $x_m^{-1}\dots x_1^{-1}$
occurs as a sequence of consecutive letters in $C$.
\end{itemize}
In case $I=\{0\}$ there are \emph{trivial} $I$-words $1_{v,\epsilon}$ for each vertex $v$ in $Q$ and $\epsilon=\pm 1$.
By a \emph{word}, we mean an $I$-word for some $I$;
it is a \emph{finite} word of \emph{length $n$} if $I=\{0,1,\dots,n\}$.
If $C$ is an $I$-word, then for each $i\in I$
there is associated a vertex $v_i(C)$, the tail of $C_i$ or the head of $C_{i+1}$,
or $v$ for $1_{v,\epsilon}$.
We say that a word $C$ is \emph{direct} or \emph{inverse} if every letter in $C$ is direct or inverse respectively.

The \emph{inverse} $C^{-1}$ of a word $C$ is defined by
inverting its letters (with $(x^{-1})^{-1} = x$) and reversing their order.
For example the inverse of a $\N$-word is a $(-\N)$-word, and vice versa.
By convention $(1_{v,\epsilon})^{-1} = 1_{v,-\epsilon}$,
and the inverse of a $\Z$-word is indexed so that
\[
(\dots C_0 \vert C_1 \dots)^{-1} = \dots C_1^{-1} \vert C_0^{-1} \dots.
\]
If $C$ is a $\Z$-word and $n\in \Z$, the \emph{shift}
$C[n]$ is the word $\dots C_{n} | C_{n+1} \dots$.
We say that a word $C$ is \emph{periodic}
if it is a $\Z$-word and $C = C[n]$ for some $n> 0$.
The minimal such $n$ is called the \emph{period}.
We extend the shift to $I$-words $C$ with $I\neq \Z$ by defining $C[n]=C$.
There is an equivalence relation $\sim$ on the set of all words
defined by $C\sim D$ if and only if $D=C[m]$ or $D=(C^{-1})[m]$ for some $m$.

\subsection*{Modules given by words}
Given any $I$-word $C$, we define a $\Lambda$-module $M(C)$ with basis $b_i$ ($i\in I$) as a vector space, 
and the action of $\Lambda$ given by
\[
e_v b_i =
\begin{cases}
b_i & \text{(if $v_i(C)=v$)} \\
0 & \text{(otherwise)}
\end{cases}
\]
for a trivial path $e_v$ in $\Lambda$ ($v$ a vertex in $Q$), and
\[
x b_i =
\begin{cases}
b_{i-1} & \text{(if $i-1\in I$ and $C_i = x$)} \\
b_{i+1} & \text{(if $i+1\in I$ and $C_{i+1} = x^{-1}$)} \\
0 & \text{(otherwise)}
\end{cases}
\]
for an arrow $x$ in $Q$.
For example, for the algebra $k[x,y]/(xy)$ and the word
\[
C = y^{-1}xxy^{-1}y^{-1}y^{-1}y^{-1}\dots,
\]
the module $M(C)$ for may be depicted as

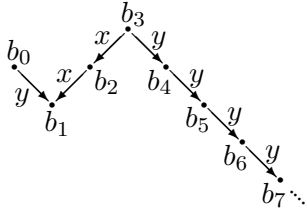
\[
\setlength{\unitlength}{0.5cm}
\begin{picture}(8,5.5)
\put(0,4){\circle*{0.15}}
\put(-0.2,4.2){$b_0$}
\put(1,3){\circle*{0.15}}
\put(0.8,2.3){$b_1$}
\put(2,4){\circle*{0.15}}
\put(2.1,3.4){$b_2$}
\put(3,5){\circle*{0.15}}
\put(2.8,5.2){$b_3$}
\put(4,4){\circle*{0.15}}
\put(3.5,3.4){$b_4$}
\put(5,3){\circle*{0.15}}
\put(4.5,2.4){$b_5$}
\put(6,2){\circle*{0.15}}
\put(5.5,1.4){$b_6$}
\put(7,1){\circle*{0.15}}
\put(6.5,0.4){$b_7$}
\put(7.3,0.7){\circle*{0.05}}
\put(7.4,0.6){\circle*{0.05}}
\put(7.5,0.5){\circle*{0.05}}
\put(7.6,0.4){\circle*{0.05}}
\put(0.1,3.9){\vector(1,-1){0.8}}
\put(0,3.1){$y$}
\put(1.9,3.9){\vector(-1,-1){0.8}}
\put(1.1,3.6){$x$}
\put(2.9,4.9){\vector(-1,-1){0.8}}
\put(2.1,4.6){$x$}
\put(3.1,4.9){\vector(1,-1){0.8}}
\put(3.6,4.6){$y$}
\put(4.1,3.9){\vector(1,-1){0.8}}
\put(4.6,3.6){$y$}
\put(5.1,2.9){\vector(1,-1){0.8}}
\put(5.6,2.6){$y$}
\put(6.1,1.9){\vector(1,-1){0.8}}
\put(6.6,1.6){$y$}
\end{picture}
\]
where the arrows show the actions of $x$ and~$y$.

For any word $C$ there is an isomorphism
$i_C : M(C)\to M(C^{-1})$
given by reversing the basis, and for a $\Z$-word $C$ and
$n\in \Z$ there is an isomorphism
$t_{C,n} : M(C) \to M(C[n])$,
given by $t_{C,n}(b_i) = b_{i-n}$.
Thus modules given by equivalent words are isomorphic.

If $C$ is a periodic word of period $n$,
then $M(C)$ becomes a $\Lambda$-$k[T,T^{-1}]$-bimodule
with $T$ acting as $t_{C,n}$, and we define 
\[
M(C,V) = M(C)\otimes_{k[T,T^{-1}]} V
\]
for $V$ a $k[T,T^{-1}]$-module. 
It is clear that $M(C)$ is free over $k[T,T^{-1}]$ of rank $n$,
so $M(C,V)$ is finite dimensional if and only if $V$ is finite dimensional.
Equivalent periodic words give rise to the same modules, since for $m\in\Z$ one has
$M(C,V) \cong M(C[m],V) \cong M((C^{-1})[m],\res_\iota V)$,
where $\iota$ is the automorphism of $k[T,T^{-1}]$ exchanging $T$ and $T^{-1}$ and $\res_\iota$ denotes the restriction map via~$\iota$.

\subsection*{String and band modules}
Let $\Lambda = kQ/(\rho)$ be a string algebra.
By a \emph{string module} we mean a module $M(C)$ with $C$ a non-periodic word,
and by a \emph{band module} we mean one of the form $M(C,V)$ with $C$ a periodic word 
and $V$ an indecomposable $k[T,T^{-1}]$-module. 
By a \emph{primitive injective} band module we mean one of the form $M(C,V)$ where $C$ is a direct or inverse periodic word
and $V$ is the injective envelope of a simple $k[T,T^{-1}]$-module.



\begin{thm}
\label{t:props}
String modules, finite-dimensional band modules and primitive injective band modules are indecomposable.
Moreover, there only exist isomorphisms between such modules when the corresponding words are equivalent: 
there are no isomorphisms between string modules and modules of the form $M(C,V)$;
string modules $M(C)$ and $M(D)$ are isomorphic if and only if $C\sim D$; and $M(C,V)\cong M(D,W)$ 
if and only if $D = C[m]$ and $W\cong V$ or $D = (C^{-1})[m]$ and $W\cong \res_\iota V$ for some $m$.
\end{thm}

Our main result is as follows.

\begin{thm}
\label{t:main}
Every finitely controlled $\Lambda$-module is isomorphic to a direct sum
of copies of string modules and finite-dimensional band modules.
\end{thm}

Note that string modules may be given by infinite words, but that not all such words give finitely controlled
or finitely generated modules. 
This is addressed in Section~\ref{s:main}.
For example the $k[x,y]/(xy)$-module $M(C)$, with $C$ as before, 
is finitely generated, while the $\Gamma$-module $M(D)$ with $\Gamma$ as above and
\[
D = \dots y_3 y_2 x_2^{-1} y_2 y_1 x_1^{-1} y_1 y_0 x_0^{-1} \vert x_1^{-1} x_2^{-1} x_3^{-1} \dots
\]
is finitely controlled, but not finitely generated. 
For the pointwise artinian case we prove the following---in fact the proof in this case is slightly easier.

\begin{thm}
\label{t:mainartinian}
Every pointwise artinian $\Lambda$-module is isomorphic to a direct sum of copies of string modules, finite-dimensional band modules 
and primitive injective band modules.
\end{thm}

Concerning uniqueness of the decomposition, we prove the following.

\begin{thm}[Krull-Remak-Schmidt property]
\label{t:KRS}
If a finitely controlled or pointwise artinian module
is written as a direct sum of indecomposables in two different ways,
then there is a bijection between the summands in such a way that
corresponding summands are isomorphic.
\end{thm}

Theorem \ref{t:mainartinian} is proved in \S\ref{s:cover}, and the others are proved in \S\ref{s:main}.
Our results extend existing work on the classification of finite-dimensional modules 
for string algebras (or related special biserial algebras) due 
to several authors \cite{GP,Ri1,DF,WW,BR}. 
These authors used the so-called functorial filtration method,
which relies on certain functorially-defined filtrations of modules.
The original work of Gelfand and Ponomarev \cite{GP}
applied to $k[x,y]/(xy)$,
and Ringel \cite{Ri1}, in what is probably the best reference for the method,
adapted it to $k\langle x,y\rangle/(x^2,y^2)$.
We modify the method so that it works for infinite-dimensional modules. 
In particular we change the definition of $C''$ for a relation $C$ in Definition~\ref{d:cprimecdprime}
and prove a Splitting Lemma, \ref{l:splitting};
we consider functors $C^\pm$ for $C$ an $\N$-word in \S\ref{s:functfilt};
we prove finite dimensionality results for refined functors in~\S\ref{s:refined}; 
we prove our Realization Lemma~\ref{l:realize} and covering properties in~\S\ref{s:cover}. 
Finally we use our Extension Theorem~\ref{t:extension} to overcome a limitation of the functorial filtration method.

Our results include the classification of finitely generated $k[x,y]/(xy)$-modules.
The possibility of such a classification is hinted at in a footnote in \cite{NRSB} (on page 652 of the English translation),
was worked out by Levy \cite{L} more generally for Dedekind-like rings, 
and again discussed by Laubenbacher and Sturmfels~\cite{LS}.
These authors all used a different method,
sometimes called `matrix reductions'.
The functorial filtration method is essentially different,
although the last part of our proof of Theorem~\ref{t:main}, 
using the Extension Theorem~\ref{t:extension}, is reminiscent of matrix reductions.
Our proof offers new insight even for the algebra $k[x,y]/(xy)$;
for example Theorem~\ref{t:summands} identifies the summands of a finitely generated module.
As discussed above, our results also give a classification
of graded modules for this algebra with finite-dimensional
homogeneous components, where $x$ and $y$ have degree 1; this appears to be new.
The same ideas would work for any grading.

Instead of a string algebra, one can
consider its localization or completion with respect to
the ideal generated by the arrows.
Algebras of this type have
occasionally been studied by matrix reductions.
For example Burban and Drozd~\cite{BD} study
the derived category for certain
`nodal' algebras, including
$k\langle\!\langle x,y\rangle\!\rangle/(x^2,y^2)$.
The functorial filtration method should adapt to
classify finitely generated modules for such
localizations and completions.
Note that Theorem~\ref{t:extension} would no longer
be necessary in this case, as there would be no
primitive simples.

\section{More about words}
\label{s:words}
We introduce some more constructions which will be needed later.
Let $\Lambda = kQ/(\rho)$ be a string algebra.
We choose a \emph{sign} $\epsilon = \pm 1$ for each letter $\ell$,
such that if distinct letters $\ell$ and $\ell'$ have the same head and sign,
then $\{\ell,\ell'\} = \{x^{-1},y\}$ for some zero relation $xy\in \rho$.
(This is equivalent to the use of $\sigma$ and $\epsilon$ in \cite{BR}.)
Note that if $C_i$ and $C_{i+1}$ are consecutive letters in a word, then $C_i^{-1}$ and $C_{i+1}$ have opposite signs.

The \emph{head} of a finite word or $\N$-word $C$ is defined to be $v_0(C)$, so it is the head of $C_1$, or $v$ for $C = 1_{v,\epsilon}$.
The \emph{sign} of a finite word or $\N$-word $C$ is defined to be that of $C_1$, or $\epsilon$ for $C = 1_{v,\epsilon}$.
The \emph{tail} is defined for a word $C$ of length $n$ to be $v_n(C)$ and for $C$ a $(-\N)$-word to be $v_0(C)$.

The \emph{composition} $CD$ of a word $C$ and a word $D$ is
obtained by concatenating the sequences of letters, provided that
the tail of $C$ is equal to the head of $D$,
the words $C^{-1}$ and $D$ have opposite signs, and
the result is a word.
By convention $1_{v,\epsilon} 1_{v,\epsilon} = 1_{v,\epsilon}$
and the composition of a $(-\N)$-word $C$ and an $\N$-word $D$ is indexed so that
\[
CD = \dots C_{-1} C_0 \vert D_1 D_2 \dots.
\]
If $C = C_1 C_2 \dots C_n$ is a non-trivial finite word and all powers $C^m$ are words,
we write $C^\infty$ and ${}^\infty C^\infty$ for the $\N$-word and periodic word
\[
C_1 \dots C_n C_1 \dots C_n C_1 \dots
\quad\text{and}\quad
\dots C_1 \dots C_n \vert C_1 \dots C_n C_1 \dots.
\]
If $C$ is an $I$-word and $i\in I$, there are words
\[
C_{>i} = C_{i+1} C_{i+2} \dots
\quad\text{and}\quad
C_{\le i} = \dots C_{i-1} C_i
\]
with appropriate conventions if $i$ is maximal or minimal in $I$, such that
\[
C = (C_{\le i} C_{>i})[-i].
\]

We say that a word $C$ is \emph{repeating} if $C = D^\infty$ for some non-trivial finite word $D$.
We say that a word $C$ is \emph{eventually repeating} (respectively \emph{direct},  respectively \emph{inverse}) 
if $C_{>i}$ is repeating (respectively direct, respectively inverse) for some $i$.
We say that an $I$-word $C$ is \emph{right vertex-finite} if for each vertex $v$ there are only finitely many $i>0$ in $I$ with $v_i(C)=v$.

\begin{lem}
\label{l:symmetry}
No word can be equal to a shift of its inverse.
\end{lem}

\begin{proof}
If $C$ is finite of length $n$, then $C=C^{-1}$
implies $C_i^{-1} = C_{n+1-i}$ for all $i$.
The same holds if $C$ is a $\Z$-word and $C = C^{-1}[-n]$.
Now if $n$ is even, then $C_i^{-1}=C_{i+1}$ for $i=n/2$,
which is impossible,
and if $n$ is odd, then $C_i^{-1} = C_i$ for $i=(n+1)/2$,
which is also impossible.
\end{proof}

\section{Primitive cycles and $k[z]$-module structure}
By a \emph{primitive cycle} $P$ we mean a non-trivial finite direct word
(so a non-trivial path in $Q$ which is non-zero in $\Lambda$)
such that ${}^\infty P^\infty$ is a periodic word of period equal to the length of $P$.
Equivalently $P$ is not itself a power of another word, and every power of $P$ is a word.
For example the primitive cycles for $k[x,y]/(xy)$ are $x$ and $y$; for $k\langle x,y\rangle/(x^2,y^2)$
they are $xy$ and $yx$; the algebra $\Gamma$ in the introduction has no primitive cycles.

A non-trivial finite direct word is uniquely determined by its first arrow and length, so there are
at most two primitive cycles with any given head~$v$.
Moreover if $P$ and $R$ are distinct primitive cycles with head $v$ the string algebra condition imples that $PR=RP=0$ in $\Lambda$. 
For any vertex $v$ we define $z_v \in e_v \Lambda e_v$ to be the sum of all primitive cycles with head $v$.
If $z_v = P+R$, then $z_v^n = P^n + R^n$ and, for example, $z_v^n P = P^{n+1}$.

Let $k[z]$ denote the polynomial ring in an indeterminate $z$. We turn any $\Lambda$-module $M$ (including $\Lambda$ itself) 
into a $k[z]$-module
by defining $z m = z_v m$ for $m \in e_v M$. The following lemma shows that this turns $\Lambda$ into a $k[z]$-algebra.

\begin{lem}
The actions of $k[z]$ and $\Lambda$ on $M$ commute.
\end{lem}

\begin{proof}
If $a$ is an arrow with head $v$ and tail $u$ then $z_v a = a z_u$, for
$z_v a$ is either zero, or it is a word of the form $P a$ where $P$ is a primitive
cycle whose first letter is $a$. Then $P a = a R$ where $R$ is a primitive cycle
at $u$, so $a R = a z_u$.
\end{proof}

\begin{lem}
$e_v \Lambda e_u$ is a finitely generated $k[z]$-module for all vertices $u,v$.
\end{lem}

\begin{proof}
Consider non-trivial paths from $u$ to $v$ in $Q$
which are non-zero in $\Lambda$.
They correspond to finite direct words $C$ with head $v$ and tail $u$.
By the string algebra condition
all such words with the same sign must
be of the form $D, PD, P^2D, \dots$ for some non-trivial words $D$ and $P$.
If there are infinitely many such words, then $P$ is a primitive cycle,
and these words are equal in $e_v \Lambda e_u$ to $D, z_v D, z_v ^2 D, \dots$.
Thus $e_v \Lambda e_u$ is a finitely generated $k[z]$-module.
\end{proof}

\begin{lem}
\label{l:kzfc}
For a $\Lambda$-module $M$, the following are equivalent.
\begin{itemize}
\item[(i)]
$M$ is finitely controlled.
\item[(ii)]
$M$ is \emph{pointwise noetherian}, meaning that for any ascending chain of submodules $M_1 \subseteq M_2 \subseteq M_3 \subseteq \dots$
and any vertex $v$ in $Q$, the chain of subspaces $e_v M_1 \subseteq e_v M_2 \subseteq e_v M_3 \subseteq \dots$ stabilizes.
\item[(iii)]
$e_v M$ is a finitely generated $e_v \Lambda e_v$-module for every vertex $v$ in $Q$.
\item[(iv)]
$e_v M$ is a finitely generated $k[z]$-module for every vertex $v$ in $Q$.
\end{itemize}
\end{lem}

\begin{proof}
(iv) $\Rightarrow$ (ii) $\Rightarrow$ (iii) $\Rightarrow$ (i) are straightforward. 
For (i) $\Rightarrow$ (iv), suppose that $M$ is finitely controlled.
Then $e_v M$ is contained in a finitely generated submodule $\sum_{i=1}^k \Lambda m_i$.
We may assume that each $m_i$ belongs to $e_{v_i} M$ for some $v_i$.
Then $e_v M$ is contained in a $k[z]$-submodule of $M$
which is isomorphic to a quotient of $\sum_{i=1}^k e_v \Lambda e_{v_i}$,
so is finitely generated as a $k[z]$-module.
\end{proof}

\begin{lem}
\label{l:kzpa}
For a $\Lambda$-module $M$, the following are equivalent.
\begin{itemize}
\item[(i)]
$M$ is  is pointwise artinian 
\item[(ii)]
$e_v M$ is an artinian $e_v \Lambda e_v$-module for every vertex $v$ in $Q$.
\item[(iii)]
$e_v M$ is an artinian $k[z]$-module for every vertex $v$ in $Q$.
\end{itemize}
\end{lem}

\begin{proof}
(iii) $\Rightarrow$ (i) $\Rightarrow$ (ii) are straightforward. 
For (ii) $\Rightarrow$ (iii),
since $e_v \Lambda e_v$ is a finitely generated $k[z]$-module, it is noetherian and its simple modules are finite dimensional.
Thus a finitely generated $e_v \Lambda e_v$-submodule of $e_v M$ is both noetherian and artinian, so
finite length, hence finite dimensional. It follows that $e_v M$ is locally finite-dimensional as a $k[z]$-module.

If $e_v \Lambda e_v$ is generated as a $k[z]$-module by $n$ elements, then there is a $k[z]$-module
map from $k[z]^n$ onto $e_v \Lambda e_v$. 
If $S$ is a simple $k[z]$-submodule of $e_v M$, tensoring with $S$, we get a map
from $S^n$ onto $e_v \Lambda e_v \otimes_{k[z]} S$, and so onto $(e_v \Lambda e_v)S$.
Thus $(e_v \Lambda e_v)S$ has length at most $n$ as a $k[z]$-module, so also as a $e_v \Lambda e_v$-module,
and hence $(e_v \Lambda e_v)S$ is contained in the $n$-th term in the socle series of $e_v M$ as an $e_v \Lambda e_v$-module.
It follows that the $k[z]$-socle of $e_v M$ is contained in the $n$-th term in the socle series of $e_v M$ as an $e_v \Lambda e_v$-module,
and as $e_v M$ is artinian as an $e_v \Lambda e_v$-module, the modules in the socle series have finite length, so they are finite dimensional.
Thus the $k[z]$-socle of $e_v M$ is finite dimensional. 

Now (iii) follows from the following characterization: 
a $k[z]$-module is artinian if and only if it is locally finite-dimensional and has finite-dimensional socle.
This follows from the fact that the injective envelopes of simple $k[z]$-modules are artinian.
\end{proof}

\section{Linear Relations}
\label{s:relations}
In this section we generalize known results about linear relations to
the infinite-dimensional case.
Let $V$ and $W$ be vector spaces.
Recall that a \emph{linear relation from $V$ to $W$} (or \emph{on $V$} if $V=W$) is a subspace
$C$ of $V\oplus W$, for example the graph of a linear map $f:V\to W$.
If $C$ is a linear relation from $V$ to $W$, $v\in V$ and $H\subseteq V$
we define
\[
Cv = \{ w\in W : (v,w)\in C \}
\quad\text{and}\quad
C H = \bigcup_{v\in H} Cv,
\]
and in this way we can think of $C$ as a mapping from elements of $V$
(or subsets of $V$) to subsets of $W$.
If $D$ is a linear relation from $U$ to $V$ then $CD$
is the linear relation from $U$ to $W$ given by
\[
CD = \{ (u,w) : \text{$\exists$ $v\in V$ with $w\in Cv$ and $v\in Du$}\}.
\]
We write $C^{-1}$ for the linear relation from $W$ to $V$ given by
\[
C^{-1} = \{ (w,v) : (v,w)\in C\},
\]
and hence we can define powers $C^n$
for all $n\in \Z$.

If $M$ is a $\Lambda$-module and $x$ is an arrow with head $v$ and tail $u$, then multiplication by
$x$ defines a linear map $e_u M \to e_v M$, and hence a linear relation from $e_u M$
to $e_v M$. 
By composing such relations and their inverses, any finite word
$C$ defines a linear relation from $e_u M$ to $e_v M$, where $v$ is the head of $C$
and $u$ is the tail of $C$. We denote this relation also by $C$.
Thus, for any subspace $U$ of $e_u M$, one obtains a subspace $C U$ of $e_v M$.
We write $C 0$ for the case $U = \{0\}$ and $C M$ for the case $U = e_u M$.
(The last makes sense if we consider $C$ as a linear relation from $M$ to itself).

\begin{dfn}
\label{d:cprimecdprime}
If $C$ is a linear relation on a vector space $V$, we define subspaces
$C'\subseteq C'' \subseteq V$ by
\begin{align*}
C'' &= \{ v\in V : \text{$\exists v_0,v_1,v_2,\dots$ with $v=v_0$ and $v_n\in C v_{n+1} \forall n$} \}, \text{ and}
\\
C' &= \bigcup_{n\ge 0} C^n 0.
\end{align*}
\end{dfn}

The first of these differs from the definition used previously,
for example in \cite{Ri1}, but that work only
involved relations on finite-dimensional vector spaces,
for which the two definitions agree:

\begin{lem}
If $C$ is a linear relation on $V$ then
\[
C'' \subseteq \bigcap_{n\ge 0} C^n V
\]
with equality if $V$ is finite-dimensional.
\end{lem}

\begin{proof}
The inclusion is clear.
If $V$ is finite-dimensional, the chain of subspaces
$V \supseteq CV \supseteq C^2 V \supseteq \dots$ stabilizes,
with $C^r V = C^{r+1}V = \dots $ for some $r$.
Then any $v\in C^r V$ belongs to $C''$ since
for any $v_n\in C^r V$ we can choose $v_{n+1}\in C^r V$ with $v_n \in C v_{n+1}$.
\end{proof}

\begin{dfn}
If $C$ is a linear relation on $V$ we define subspaces $C^\flat \subseteq C^\sharp \subseteq V$ by
$C^\sharp = C'' \cap (C^{-1})''$ 
and
$C^\flat = C'' \cap (C^{-1})' + C' \cap (C^{-1})''$.
(Note the symmetry between $C$ and $C^{-1}$.)
\end{dfn}

\begin{lem}
\label{l:sharpflatprops}
(i) $C^\sharp  \subseteq C C^\sharp $,
(ii) $C^\flat  = C^\sharp \cap CC^\flat $,
(iii) $C^\sharp  \subseteq C^{-1}C^\sharp $, and
(iv) $C^\flat  = C^\sharp \cap C^{-1}C^\flat $.
\end{lem}

\begin{proof}
(i) If $v\in C^\sharp $ then there are $v_n$ ($n\in\Z$) with $v_0=v$,
$v_n\in C v_{n+1}$ for all $n$.
Now $v \in C v_1$ and clearly $v_1\in C^\sharp $,
so $C^\sharp  \subseteq C C^\sharp $.

(ii) Suppose $b\in C^\flat $. We write it as $b = b^+ + b^-$
with $b^+ \in C'' \cap (C^{-1})'$ and $b^- \in C' \cap (C^{-1})''$.
Now there are $b^\pm _n$ ($n\in\Z$) with
$b^\pm = b^\pm_0$, $b^\pm_n \in C b^\pm_{n+1}$ for all $n$,
$b^+_n = 0$ for $n\ll 0$ and $b^-_n = 0$ for $n \gg 0$.
Clearly $b^+_1 + b^-_1 \in C^\flat $ and $b = b^+ + b^- \in C(b^+_1 + b^-_1)$,
so $C^\flat \subseteq C^\sharp \cap CC^\flat $.
Conversely, suppose that $v \in C^\sharp \cap C b$.
Then $b^\pm_{-1} \in C b^\pm$, so
\[
v - b^+_{-1} - b^-_{-1}
\in C^\sharp \cap C(b-b^+-b^-)
= C^\sharp \cap C0
\subseteq C^\sharp \cap C' \subseteq C^\flat.
\]
Clearly also $b^\pm_{-1}\in C^\flat $, so $v\in C^\flat$.

(iii) and (iv) follow by symmetry between $C$ and $C^{-1}$.
\end{proof}

\begin{lem}
A linear relation $C$ on $V$ induces an automorphism $\theta$ of $C^\sharp/C^\flat$
with $\theta(C^\flat + v) = C^\flat + w$  if and only if $w \in C^\sharp \cap (C^\flat + Cv)$.
\end{lem}

\begin{proof}
For $v\in C^\sharp$ we define $\theta$ by $\theta(C^\flat + v) = C^\flat + w$
where $w$ is any element of $C^\sharp \cap (C^\flat + Cv)$.
There always is some $w$ by Lemma~\ref{l:sharpflatprops}(iii),
and $\theta$ is well-defined since if $w'\in C^\sharp \cap (C^\flat + Cv')$
and $v-v'\in C^\flat$, then
\[
w-w' \in C^\sharp \cap (C^\flat + C(v-v')) \subseteq C^\flat  + C^\sharp\cap C(v-v')
\subseteq C^\flat  + C^\sharp\cap C C^\flat = C^\flat
\]
by Lemma~\ref{l:sharpflatprops}(ii).
Clearly $\theta$ is a linear map, and by symmetry between $C$ and $C^{-1}$ it is an automorphism.
\end{proof}

If $C$ is a linear relation on $V$, we say that $C$ is \emph{split} if there is a subspace $U$ of $V$ such that
$C^\sharp =C^\flat \oplus U$ and the restriction of $C$ to $U$ is an automorphism.

\begin{lem}[Splitting Lemma]
\label{l:splitting}
If $C$ is a linear relation on $V$ and $C^\sharp/C^\flat$ is finite-dimensional, then $C$ is split.
\end{lem}

\begin{proof}
Let $\theta$ be the induced automorphism of $C^\sharp/C^\flat$
and let $A = (a_{ij})$ be the matrix of $\theta$ with respect to
a basis $C^\flat +v_1,\dots , C^\flat +v_k$ of $C^\sharp /C^\flat$.
Thus
\[
\theta(C^\flat  + v_j) = \sum_{i=1}^k a_{ij} (C^\flat  + v_i) =  C^\flat  + \sum_{i=1}^k a_{ij} v_i
\]
so there are $b_1,\dots b_k\in C^\flat $ with
\[
b_j + \sum_{i=1}^k a_{ij} v_i  \in C v_j
\]
for all $j$.
We write $b_j = b^+_j + b^-_j$ with $b^+_j \in C'' \cap (C^{-1})'$
and $b^-_j \in C' \cap (C^{-1})''$. Now there are $b^\pm _{j,n}$ ($n\in\Z$)
with
$b^\pm_j = b^\pm_{j,0}$, $b^\pm_{j,n} \in C b^\pm_{j,n+1}$ for all $n$,
$b^+_{j,n} = 0$ for $n\ll 0$ and $b^-_{j,n} = 0$ for $n \gg 0$.
Define matrices $M^{\pm,n} = (m^{\pm,n}_{i,j})$ for $n\in \Z$ by
\[
M^{+,n} =
\begin{cases}
0 & (n > 0) \\
(A^{-1})^{1-n} & (n\le 0)
\end{cases}
\quad\text{and}\quad
M^{-,n} =
\begin{cases}
- A^{n-1} & (n>0) \\
0 & (n\le 0).
\end{cases}
\]
and let
\[
u_j = v_j
+ \sum_{n\in\Z} \sum_{i=1}^k m^{+,n}_{ij} b^+_{i,n}
+ \sum_{n\in\Z} \sum_{i=1}^k m^{-,n}_{ij} b^-_{i,n}.
\]
These are finite sums since $M^{+,n}=0$ for $n> 0$
and $b^+_{i,n}=0$ for $n\ll 0$,
and $M^{-,n}=0$ for $n\le 0$
and $b^-_{i,n}=0$ for $n\gg 0$.
Now
\[
b_j +\sum_{i=1}^k a_{ij} v_i  \in C v_j
\]
implies
\[
b_{j,0}
+ \sum_{i=1}^k a_{ij} v_i
+ \sum_{n\in\Z} \sum_{i=1}^k m^{+,n}_{ij} b^+_{i,n-1}
+ \sum_{n\in\Z} \sum_{i=1}^k m^{-,n}_{ij} b^-_{i,n-1}
\in C u_j.
\]
If $\delta_{pq}$ is the Kronecker delta function, we have
\[
\delta_{n0} I + M^{\pm,n+1} = AM^{\pm,n}
\]
which enables this to be rewritten as
\[
\sum_{i=1}^k a_{ij} u_i \in C u_j
\]
for all $j$.
Then $C^\sharp =C^\flat \oplus U$ where
$U$ has basis $u_1,\dots,u_k$,
and $C$ induces on $U$ the automorphism with matrix $A$.
\end{proof}

\section{Torsion}
\label{s:torsion}
A $k[z]$-module $V$ is torsion if and only if it is locally finite dimensional.
The torsion submodule $\tau(V)$ of an arbitrary module $V$ decomposes as 
the direct sum of
\begin{align*}
\tau^0(V) &= \{ v\in V : \text{$z^n v=0$ for some $n\ge 0$} \}, \text{ and}
\\
\tau^1(V) &= \{ v\in V : \text{$f(z)v=0$ for some $f(z)\in k[z]$ with $f(0)=1$} \}
\end{align*}
which we call the \emph{nilpotent torsion} and \emph{primitive torsion} submodules of $V$.

\begin{lem}
\label{l:splittingb}
If $V$ is a torsion $k[z]$-module, and $C = \{ (v,zv) : v\in V\}$ is the graph of multiplication by $z$,
then $C$ is a split relation.
\end{lem}

\begin{proof}
Multiplication by $z$ is invertible on $\tau^1(V)$,  so $\tau^1(V) \subseteq C''$. Also $C' = 0$, $(C^{-1})'' = V$ and $(C^{-1})' = \tau^0 (V)$.
Thus 
\[
C^\sharp = C'' \cap (\tau^0(V)\oplus \tau^1(V)) = (C''\cap \tau^0(V)) \oplus \tau^1(V) = C^\flat \oplus \tau^1(V).
\qedhere
\]
\end{proof}

Now we return to the string algebra $\Lambda = kQ/(\rho)$.
If $M$ is a $\Lambda$-module, we consider it as a $k[z]$-module,
and hence define $\tau(M)$, $\tau^0(M)$ and $\tau^1(M)$.
They are $\Lambda$-submodules of $M$, and we have
\[
\tau(M) = \bigoplus_v \tau(e_v M)
\quad\text{and}\quad
\tau^i(M) = \bigoplus_v \tau^i(e_v M).
\]
We say that $M$ is \emph{nilpotent torsion} if $M=\tau^0(M)$
and \emph{primitive torsion} if $M=\tau^1(M)$.
If $P$ is a primitive cycle with head $v$ we can also consider $e_v M$
as a $k[P]$-module, and we write
\[
\tau_P(e_v M) = \tau_P^0(e_v M) \oplus \tau_P^1(e_v M)
\]
for the corresponding torsion submodules. 
They are $k[z]$-submodules of $e_v M$.

\begin{lem}
\label{l:tauprimed}
We have
\[
\tau^0(e_v M) = \bigcap_P \tau_P^0(e_v M)
\quad\text{and}\quad
\tau^1(e_v M) = \bigoplus_P \tau_P^1(e_v M)
\]
where $P$ runs through the (up to two) primitive cycles with head $v$.
\end{lem}

\begin{proof}
We only need to deal with the case when there are two primitive cycles $P,R$ with head $v$.
If $m\in \tau^0(e_v M)$ then $z^n m=0$ for some $n$. Thus $(P^n + R^n)m=0$,
so $P^{n+1}m = R^{n+1}m = 0$, and hence $m \in \tau_P^0(e_v M)\cap \tau_R^0(e_v M)$.
Also $\tau_P^1(e_v M)$ is annihilated by $R$, so its intersection with $\tau_R^1(e_v M)$ must be zero.
Now suppose that $m\in e_v M$ and $f(z)m = 0$ with $f(z)=1+g(z)$ where $g(0)=0$.
Then $0 = f(P+R)m = m + g(P)m + g(R)m$.
Thus $0 = g(P)(m + g(P)m + g(R)m) = g(P)m + g(P)^2 m = f(P) g(P)m$,
so $g(P)m\in \tau_P^1(e_v M)$. Similarly $g(R)m\in \tau_R^1(e_v M)$,
so $m = -g(P)m - g(R)m$ is in the direct sum.
\end{proof}

\begin{lem}
\label{l:ptorsion}
Suppose $P$ is a primitive cycle with head $v$ and $M$ is a $\Lambda$-module. 
Let $I = \bigcap_{n\ge 0} P^n M$.
\begin{itemize}
\item[(i)]
If $M$ is finitely controlled, then $I = P'' = \tau_P^1(e_v M)$.
\item[(ii)]
If $M$ is finitely controlled or pointwise artinian, then $I \subseteq PI$.
\end{itemize}
\end{lem}

\begin{proof}
(i) Clearly $\tau_P^1(e_v M) \subseteq P'' \subseteq I$.
Now $e_v M$ is a finitely
generated module for the ring $k[P]$, or $k[P,R]/(PR)$ if there is another primitive cycle $R$ with head $v$.
Then by Krull's Theorem~\cite[Theorem 8.9]{Matsumura} applied to $e_v M$
and the ideal generated by $P$, we have 
$I \subseteq \tau^1_P(e_v M)$.

(ii) The finitely controlled case follows from (i). 
The subspaces $P^n M$ are $k[z]$-submodules, so in the pointwise artinian case 
we have $P^n M = P^{n+1}M = \dots$ for some $n$, so $I = P^{n+1}M = PI$.
\end{proof}

\section{Functorial filtration given by words}
\label{s:functfilt}
For $v$ a vertex and $\epsilon=\pm 1$, we define $\W_{v,\epsilon}$ to be the set of
all words with head $v$ and sign $\epsilon$. They are necessarily either 
finite words or $\N$-words.
There is a total order on $\W_{v,\epsilon}$ given by $C < C'$ if
\begin{itemize}
\item[(a)]$C = ByD$ and $C' = Bx^{-1}D'$ where $B$ is a finite word, $x,y$ are arrows, and $D,D'$ are words, or
\item[(b)]$C'$ is a finite word and $C = C'yD$ where $y$ is an arrow and $D$ is a word, or
\item[(c)]$C$ is a finite word and $C' = Cx^{-1}D'$ where $x$ is an arrow and $D'$ is a word.
\end{itemize}

For any $\Lambda$-module $M$ and $C\in \W_{v,\epsilon}$ we define subspaces
\[
C^-(M) \subseteq C^+(M) \subseteq e_v M
\]
as follows.
First suppose that $C$ is a finite word.
Then $C^+(M) = C x^{-1}0$ if there is an
arrow $x$ such that $C x^{-1}$ is a word, and otherwise $C^+(M) = C M$.
Similarly, $C^-(M) = C y M$ if there is an arrow $y$ such that $C y$ is a word, and
otherwise $C^-(M) = C 0$.
Now suppose that $C$ is an $\N$-word.
Then $C^+(M)$ is the set of $m\in M$ such that
there is a sequence $m_n$ ($n\ge 0$) such that $m_0=m$
and $m_{n-1} \in C_n m_n$ for all $n$.
One defines $C^-(M)$ to be the set of $m\in M$
such that there is a sequence $m_n$ as above which is eventually zero.
Equivalently $C^-(M) = \bigcup_n C_{\le n} 0$.
Observe that if $C\in \W_{v,\epsilon}$ is repeating,
say $C=D^\infty$, then $C^-(M) = D'$ and $C^+(M) = D''$,
where $D$ is considered as a linear relation on $e_v M$.

Clearly one has $\theta(C^\pm(M))\subseteq C^\pm(N)$
for a homomorphism $\theta:M\to N$ of $\Lambda$-modules.
Thus $C^\pm$ define subfunctors of the forgetful
functor from $\Lambda$-modules to vector spaces (or $k[z]$-modules).

\begin{lem}
\label{l:dirsums}
The functors $C^\pm$ commute with arbitrary direct sums.
\end{lem}

\begin{proof}
Straightforward.
\end{proof}

\begin{lem}
\label{l:onesideorder}
If $C,D\in \W_{v,\epsilon}$ and $C<D$, then $C^+(M) \subseteq D^-(M)$.
\end{lem}

\begin{proof}
Standard. For finite words, see the lemma on page 23 of \cite{Ri1}.
\end{proof}

\section{Refined Functors}
\label{s:refined}
If $B$ and $D$ are words with the same head $v$ and opposite signs, and $M$ is a $\Lambda$-module, we define
\begin{align*}
F_{B,D}^+(M) &= B^+(M) \cap D^+(M),
\\
F_{B,D}^-(M) &= (B^+(M) \cap D^-(M)) + (B^-(M) \cap D^+(M)),\text{ and}
\\
F_{B,D}(M) &= F_{B,D}^+(M)/F_{B,D}^-(M).
\end{align*}

If $C = B^{-1}D$ is a non-periodic word,
we consider $F_{B,D}$ as a functor from the category of $\Lambda$-modules
to vector spaces. 

If $C = B^{-1}D$ is a periodic word, say of period $n$, then $C = {}^\infty E^\infty$ for some word $E$ of length $n$ and head $v$. 
If $M$ is a $\Lambda$-module,
then $E$ induces a linear relation on $e_v M$, and
$F_{B,D}^+(M) = E^\sharp $ and
$F_{B,D}^-(M) = E^\flat $
as in Section~\ref{s:relations}.
Thus $E$ induces an automorphism of $F_{B,D}(M) = E^\sharp/E^\flat$, 
and hence $F_{B,D}$ defines a functor from $\Lambda$-modules to $k[T,T^{-1}]$-modules, with the action of $T$ given by this automorphism.
We say that $M$ is \emph{$E$-split} or \emph{$C$-split} if the relation $E$ on $e_v M$ is split.

Let $v$ be a vertex.
If $(B,D) \in \W_{v,1}\times \W_{v,-1}$ and $M$ is a $\Lambda$-module, we define
\[
G^\pm_{B,D}(M) = B^-(M) + D^\pm(M) \cap B^+(M) \subseteq e_v M
\]
Clearly $G^-_{B,D}(M) \subseteq G^+_{B,D}(M)$ and $G^+_{B,D}(M)/G^-_{B,D}(M) \cong F_{B,D}(M)$.
We totally order $\W_{v,1}\times \W_{v,-1}$ lexicographically, so
\[
(B,D) < (B',D') \quad \Leftrightarrow \quad \text{if $B < B'$ or ($B=B'$ and $D<D'$).}
\]
We have $G^+_{B,D}(M) \subseteq G^-_{B',D'}(M)$ for $(B,D)<(B',D')$ by Lemma~\ref{l:onesideorder}.

\begin{lem}
\label{l:shiftfunctor}
~
\begin{itemize}
\item[(i)]
\vspace{-4pt}
$F_{B,D}$ commutes with direct sums.
\item[(ii)]
If $B^{-1}D$ is not a word, then $F_{B,D} = 0$.
\item[(iii)]
If $B^{-1}D$ is a non-periodic word, then $F_{D,B} \cong F_{B,D}$.
\item[(iv)]
If $B^{-1}D$ is a periodic word, then $F_{D,B} \cong \res_\iota \, F_{B,D}$.
\item[(v)]
If $C$ is a fixed word, the functors $F_{B,D}$ with $B^{-1}D = C[n]$, for any $n$,
are all isomorphic.
\end{itemize}
\end{lem}

\begin{proof}
(i) This follows from Lemma~\ref{l:dirsums}.

(ii) $B^{-1}D$ must involve a zero relation, and exchanging $B$ and $D$ if necessary, we may assume that
$B = x_n^{-1}\dots x_1^{-1} C$ and $D = y_1\dots y_r E$ with $x_1\dots x_n y_1\dots y_r\in \rho$.
If $m\in F_{B,D}^+(M)$ then $m = y_1\dots y_r m'$ with $m'\in E^+(M)$,
so $m\in x_n^{-1}\dots x_1^{-1} 0 \subseteq B^-(M)$, so $m\in F_{B,D}^-(M)$.

(iii), (iv) Clear.

(v) This is the same as the corresponding lemma at the top of page 25 in \cite{Ri1}.
The extension to functors to $k[T,T^{-1}]$-modules in case $C$
is periodic is straightforward.
\end{proof}

\begin{lem}
\label{l:fdintersection}
Suppose $B$ and $D$ are non-trivial words with head $v$ and opposite signs, and that the first letters of both are direct.
If $M$ is finitely controlled or pointwise artinian then $B^+(M)\cap D^+(M)$ is finite dimensional.
\end{lem}

\begin{proof}
The action of $z$ annihilates $B^+(M)\cap D^+(M)$, since any arrow with tail $v$ has zero composition with the first arrow of $B$ or $D$.
Now since $e_v M$ is either finitely generated or artinian as a $k[z]$-module, 
the subspace $\{ m\in e_v M : zm=0\}$ is finite dimensional.
\end{proof}

\begin{lem}
\label{l:fbdsplit}
Suppose that $M$ is a finitely controlled or pointwise artinian $\Lambda$-module.
Suppose that $C = B^{-1}D$ is periodic of period $n$,
so $C = {}^\infty E^\infty$ for some word $E$ of length $n$ and head $v$, and $B = (E^{-1})^\infty$ and $D = E^\infty$.
Then either
\begin{itemize}
\item[(i)]
$F_{B,D}(M)$ is finite dimensional, or
\item[(ii)]
$F_{B,D}(M)$ is an artinian $k[T,T^{-1}]$-module and $E$ or $E^{-1}$ is a primitive cycle.
\end{itemize}
In either case, the relation $E$ on $e_v M$ is split.
\end{lem}

\begin{proof}
If $C$ is not direct or inverse, then by Lemma~\ref{l:shiftfunctor} we may apply a shift, and hence we may suppose that the situation 
of Lemma~\ref{l:fdintersection} applies, so case (i) holds.
Supposing otherwise, and interchanging $B$ and $D$ if necessary, we may suppose that $C$ is direct, so since it is periodic, $E = P$, a primitive cycle.
Now if $M$ is finitely controlled, we have $F_{B,D}^+(M) = P'' = \tau^1_P(e_v M)$ by Lemma~\ref{l:ptorsion}, which is finite dimensional.
If not, then $F_{B,D}(M)$ is a quotient of a $k[z]$-submodule of $e_v M$, with the action of $T$ being the same as the action of $z$,
so it is artinian. 
Now the splitting follows from the Splitting Lemma~\ref{l:splitting} in case (i) or Lemma~\ref{l:splittingb} in case (ii).
\end{proof}

\section{Evaluation on string and band modules}
The results in this section are essentially the same as those in \cite[\S\S 4,5]{Ri1}.
Suppose $C$ is an $I$-word.
For $i\in I$, the words $C_{>i}$ and $(C_{\le i})^{-1}$
have head $v_i(C)$ and opposite signs.
For $\epsilon=\pm 1$, let $C(i,\epsilon)$ denote the one which has sign $\epsilon$.
We define $d_i(C,\epsilon) = 1$ if $C(i,\epsilon) = C_{>i}$ and
$d_i(C,\epsilon) = -1$ if $C(i,\epsilon) = (C_{\le i})^{-1}$.

\subsection*{String modules}
Recall that if $C$ is a non-periodic $I$-word, the string module $M(C)$ has basis the symbols $b_i$ for $i\in I$.

\begin{lem}
\label{l:evword}
If $D\in \W_{v,\epsilon}$ then
\begin{itemize}
\item[(i)]
$D^+(M(C))$ has basis $\{ b_i : v_i(C)=v, C(i,\epsilon) \le D \}$, and
\item[(ii)]
$D^-(M(C))$ has basis $\{ b_i : v_i(C)=v, C(i,\epsilon) < D \}$.
\end{itemize}
\end{lem}

\begin{proof}
Let $M = M(C)$.
Using the ordering on words and functors, it suffices to show that
$b_i \in C(i,\epsilon)^+(M)$ and that if a linear combination $m$
of the basis elements $b_j$ belongs to $C(i,\epsilon)^-(M)$,
then the coefficient of $b_i$ in $m$ is zero.

If $C(i,\epsilon)$ is finite, let $1_{u,\eta}$ be the trivial word with $C(i,\epsilon)1_{u,\eta}$
defined (and hence equal to $C(i,\epsilon)$).
Define $d = d_i(C,\epsilon)$.
For $n\ge 1$, and not greater than the length
of $C(i,\epsilon)$, we have $b_{i+d(n-1)} \in C(i,\epsilon)_n b_{i+dn}$.
Moreover, if $C(i,\epsilon)$ has length $n$ then $b_{i+dn}\in 1_{u,\eta}^+(M)$.
It follows that $b_i \in C(i,\epsilon)^+(M)$.

By induction on $n$, the following is straightforward.
Suppose $n$ is not greater than the length of $C(i,\epsilon)$.
If $m$ is an element of $M$ whose coefficient of $b_i$ is $\lambda$,
and $m\in C(i,\epsilon)_{\le n} m'$, then the coefficient
of $b_{i+dn}$ in $m'$ is also $\lambda$.
Clearly if $C(i,\epsilon)$ has length $n$, then no element of $1_{u,\eta}^-(M)$
has $b_{i+dn}$ occuring with non-zero coefficient.
It follows that no element of $C(i,\epsilon)^-(M)$
can have $b_i$ occuring with non-zero coefficient.
\end{proof}

\begin{lem}
\label{l:evtwostring}
Let $M=M(C)$ where $C$ is a non-periodic $I$-word.
\begin{itemize}
\item[(i)]
If $i\in I$, then $F_{C(i,1),C(i,-1)}^+(M) = F_{C(i,1),C(i,-1)}^-(M) \oplus k b_i$.
\item[(ii)]
If $B^{-1}D = C$, then $F_{B,D}(M) \cong k$.
\item[(iii)]
If $B^{-1} D$ is not equivalent to $C$, then $F_{B,D}(M)=0$.
\end{itemize}
\end{lem}

\begin{proof}
(i) By Lemma~\ref{l:evword},
\[
F_{C(i,1),C(i,-1)}^+(M) = F_{C(i,1),C(i,-1)}^-(M) \oplus U
\]
where $U$ is spanned the $b_j$ with $C(j,1) = C(i,1)$ and $C(j,-1) = C(i,-1)$.
By Lemma \ref{l:symmetry}, and
since $C$ is not periodic, this condition holds only for $j=i$.

(ii) We have $\{B,D\} = \{C(i,1),C(i,-1)\}$ for some $i$.

(iii) Exchanging $B$ and $D$ if necessary, and letting $v$ be the head of $B$ and $D$, we have $(B,D)\in \W_{v,1}\times\W_{v,-1}$.
Lemma~\ref{l:evword} implies that the spaces $G_{B,D}^\pm (M)$
are spanned by sets of basis elements $b_j$, so
if $F_{B,D}(M)\neq 0$, then some $b_i$ belongs to
$G_{B,D}^+(M)$ but not to $G_{B,D}^-(M)$.
But by (i) we have
\[
b_i \in G_{C(i,1),C(i,-1)}^+(M) \setminus G_{C(i,1),C(i,-1)}^-(M).
\]
Then $(B,D) = (C(i,1),C(i,-1))$ by the total ordering of the $G^\pm_{B,D}$,
so $B^{-1}D$ is equivalent to $C$.
\end{proof}

\begin{lem}
\label{l:mapstring}
Suppose that $C$ is a non-periodic $I$-word.
Suppose that $i\in I$, $B = C(i,1)$ and $D = C(i,-1)$.
Let $M$ be a $\Lambda$-module and consider
$M(C) \otimes_k F_{B,D}(M)$ as a direct sum of copies of $M(C)$ indexed 
by a (possibly infinite) basis of $F_{B,D}(M)$.
Then there is a map $\theta_{B,D,M}:M(C) \otimes_k F_{B,D}(M) \to M$ such
that $F_{B,D}(\theta_{B,D,M})$ is an isomorphism.
\end{lem}

\begin{proof}
Take a basis $(f_\lambda)$ of $F_{B,D}(M)$, and lift the elements 
$f_\lambda$  to elements $m_\lambda \in F_{B,D}^+(M) = B^+(M)\cap D^+(M)$.
In each case there is a $\Lambda$-module map $\theta_\lambda : M(C)\to M$ sending $b_i$ to $m_\lambda$.
These combine to give a map $\theta_{B,D,M}:M(C) \otimes_k F_{B,D}(M) \to M$.
By Lemma~\ref{l:evtwostring}, the map $F_{B,D}(\theta_{B,D,M})$ is an isomorphism.
\end{proof}

\subsection*{Band modules}
Suppose that $C$ is a periodic word of period $n$ and $V$ is a $k[T,T^{-1}]$-module.
The module $M(C,V) = M(C)\otimes_{k[T,T^{-1}]} V$ can be written as
\[
M(C,V) = V_0 \oplus V_1 \oplus \dots \oplus V_{n-1}
\]
where each $V_i = b_i\otimes V$ is identified with a copy of $V$.
(It is a band module provided $V$ is indecomposable.)

\begin{lem}
\label{l:bandeval}
If $D\in \W_{v,\epsilon}$ then
\begin{itemize}
\item[(i)]
$D^+(M) = \bigoplus_{i\in I^+} V_i$,\ \ $I^+ = \{ 0\le i < n : v_i(C)=v, C(i,\epsilon) \le D \}$,
\item[(ii)]
$D^-(M) = \bigoplus_{i\in I^-} V_i$,\ \ $I^- = \{ 0\le i < n : v_i(C)=v, C(i,\epsilon) < D \}$.
\end{itemize}
\end{lem}

\begin{proof}
Similar to Lemma~\ref{l:evword}.
\end{proof}

\begin{lem}
\label{l:evtwoband}
Let $M = M(C,V)$.
\begin{itemize}
\item[(i)]
If $0\le i<n$, then $F_{C(i,1),C(i,-1)}^+(M) = F_{C(i,1),C(i,-1)}^-(M) \oplus V_i$.
\item[(ii)]
If $B^{-1}D = C$, then $F_{B,D}(M)\cong V$ as $k[T,T^{-1}]$-modules.
\item[(iii)]
If $B^{-1} D$ is not equivalent to $C$ then $F_{B,D}(M(C,V))=0$.
\end{itemize}
\end{lem}

\begin{proof}
Similar to Lemma~\ref{l:evtwostring}.
\end{proof}

\begin{lem}
\label{l:mapband}
Suppose that $C = B^{-1}D$ is a periodic word and $M$ is a $C$-split module.
Let $V = F_{B,D}(M)$. Then there is a homomorphism $\theta_{B,D,M}:M(C,V)\to M$ such
that $F_{B,D}(\theta_{B,D,M})$ is an isomorphism.
\end{lem}

\begin{proof}
We have $D = E^\infty$ and $B = (E^{-1})^\infty$.
Then $V = E^\sharp /E^\flat $, and as a $k[T,T^{-1}]$-module, the action of $T$ is induced by $E$.
By assumption $E^\sharp  = E^\flat \oplus U$, such that $E$ induces
an automorphism on $U$, and of course $U\cong V$.
As in \cite[\S5, Proposition]{Ri1},
one gets a mapping $\theta_{B,D,M}:M(C,V)\to M$ such that
$F_{B,D}(\theta_{B,D,M})$ is an isomorphism.
Namely, there are elements
$u_{r,i}\in M$ for $1\le r\le s$ and $0\le i\le n$
with $u_{1,0},\dots,u_{s,0}$ and $u_{1,n},\dots,u_{s,n}$
bases of $U$ connected by $u_{r,0} = T u_{r,n}$,
and $u_{r,i-1}\in E_i u_{ri}$ for all $r,i$.
Using these elements one defines $\theta_{B,D,M}:M(C,U)\to M$,
sending $b_i\otimes \overline u_{r,0} \in V_i$ for $0\le i<n$
to $u_{r,i}$.
\end{proof}

\section{Direct sums of string and band modules}
\label{s:dirsums}
\begin{thm}
\label{t:summands}
Suppose that $M$ is a direct sum of copies of string modules and modules of the form $M(C,V)$, say
\[
M = \left(\bigoplus_\lambda M(C^{\lambda})\right) \oplus \left( \bigoplus_\mu M(C^\mu,V^\mu) \right).
\]
\begin{itemize}
\item[(i)]
If $B^{-1}D$ is a non-periodic word,
then $\dim F_{B,D}(M)$ is equal to the number of string module summands $M(C^\lambda)$ with $C^\lambda \sim B^{-1}D$.
\item[(ii)]
If $B^{-1}D$ is a periodic word, then $\dim F_{B,D}(M)$ is isomorphic to the direct sum of the $V^\mu$
for $\mu$ such that $C^\mu$ a shift of $B^{-1}D$ and of $\res_\iota V^\mu$ for $\mu$ such that $C^\mu$ is a shift of $D^{-1}B$.
\end{itemize}
\end{thm}

\begin{proof}
Follows immediately from Lemmas~\ref{l:evtwostring} and \ref{l:evtwoband}.
\end{proof}

If $V$ is a finite-dimensional (respectively artinian) $k[T,T^{-1}]$-module, 
we can write it as a finite direct sum of indecomposables $V=V_1\oplus\dots\oplus V_n$,
where the summands are finite dimensional (respectively finite dimensional or injective envelopes of simple modules). 
Thus if $C$ is a periodic word (respectively a direct or inverse periodic word) 
we can write $M(C,V) \cong M(C,V_1)\oplus\dots \oplus M(C,V_n)$, a direct sum of finite-dimensional band modules 
(respectively finite dimensional or primitive injective band modules).

\begin{thm}
\label{t:existdecomp}
Suppose that $M$ is a finitely controlled (respectively pointwise artinian) $\Lambda$-module.
Then there is a homomorphism $\theta:N\to M$ where 
$N$ is a direct sum of string and finite-dimensional band modules 
(respectively a direct sum of string, finite-dimensional band modules and primitive injective band modules)
with the property that $F_{B,D}(\theta)$ is an isomorphism for all refined functors $F_{B,D}$.
\end{thm}

\begin{proof}
If $C = B^{-1}D$ is a non-periodic word, then Lemma \ref{l:mapstring} gives a map $\theta_{B,D,M}$
from a direct sum of copies of $M(C)$ to $M$.
If $C = B^{-1}D$ is periodic word, then $M$ is $C$-split by Lemma~\ref{l:fbdsplit},
and Lemma \ref{l:mapband} gives a map $\theta_{B,D,M}$ from a module of the form $M(C,V)$ to $M$. 
As indicated above, we can decompose $M(C,V) \cong M(C,V_1)\oplus\dots \oplus M(C,V_n)$, 
a direct sum of finite-dimensional band modules 
(or finite dimensional and primitive injective band modules).
Let $N$ be the direct sum of all of these string and band modules 
as $(B,D)$ runs through pairs in such a way that $C=B^{-1}D$ runs through the equivalence classes of words, once each.
The maps $\theta_{B,D,M}$ combine to give a map $\theta:N\to M$ with $F_{B,D}(\theta)$ an isomorphism
for all these pairs, and hence for any pair of words $B,D$ with the same head and opposite signs.
\end{proof}

\begin{lem}
\label{l:contprimtor}
Suppose $\theta:N\to M$ is a homomorphism, with $M$ finitely controlled
and such that $F_{B,D}(\theta)$ is an isomorphism for all refined functors $F_{B,D}$. 
Then $\Ima(\theta)$ contains the primitive torsion submodule $\tau^1(M)$ of $M$.
\end{lem}

\begin{proof}
By Lemma~\ref{l:tauprimed} it suffices to show
$\tau_P^1(e_v M)\subseteq \Ima(\theta)$ for $P$ a primitive cycle with head $v$.
Let $m\in \tau_P^1(e_v M)$. By Lemma~\ref{l:ptorsion},
\[
m\in P'' = P'' \cap (P^{-1})'' = F_{B,D}^+(M)
\]
where $B = (P^{-1})^\infty$ and $D = P^\infty$.
Thus by hypothesis $m = m'+\theta(n)$ for some $n\in N$
and
\[
m'\in F_{B,D}^-(M) = (P' \cap (P^{-1})'') +(P''\cap (P^{-1})').
\]
Now $P'=0$ since $P$ is direct, and $P''\cap (P^{-1})' = \tau_P^1(e_v M) \cap \tau_P^0(e_v M) = 0$.
Thus $m'=0$, so $m=\theta(n)\in\Ima(\theta)$.
\end{proof}

\begin{lem}
\label{l:sbmono}
Suppose $\theta:N\to M$ is a homomorphism, with $N$ a direct sum of string and band modules
and such that $F_{B,D}(\theta)$ is an isomorphism for all refined functors $F_{B,D}$. 
Then $\theta$ is injective.
\end{lem}

\begin{proof}
Suppose that $n$ is a non-zero element of $e_v N$ with $\theta(n) = 0$. 
We can write $n$ as a sum of components in different summands of $N$. 
Let $S$ be one of these summands.
If $S$ is a string module, the component can be written as a linear combination of the basis elements $b_i$, 
and if $S$ is a band module $M(C,V)$, the component can be written as a sum of elements in the vector spaces $V_i$. 
By Lemmas~\ref{l:evtwostring} and~\ref{l:evtwoband}, 
there is $(B,D) \in \W_{v,1}\times \W_{v,-1}$
with $F_{B,D}^+(S) = F_{B,D}^-(S)\oplus U$
where $U = kb_i$ or $V_i$. 
It follows that $G_{B,D}^+(S) = G_{B,D}^-(S)\oplus U$.
Only finitely many $b_i$ and $V_i$ from finitely many summands $S$ of $N$ 
make a non-zero contribution to $n$, and among the finitely many pairs $(B,D)$ 
which arise, choose $B$ maximal, and for the pairs with this $B$, choose $D$ maximal. 
Then $n$ is in $G_{B,D}^+(N)$ but not in $G_{B,D}^-(N)$. 
But this means that $n$ induces a non-zero element of $F_{B,D}(N)$.
Thus by assumption $\theta(n)$ induces a non-zero element of $F_{B,D}(M)$.
But this is impossible since $\theta(n)=0$.
\end{proof}

\section{Covering property}
\label{s:cover}
\begin{lem}
\label{l:cplustrunc}
Let $C$ be an $\N$-word and $M$ a $\Lambda$-module. If
\begin{itemize}
\item[(i)] $M$ is pointwise artinian, or
\item[(ii)] $M$ is finitely controlled and $C$ is not (direct and repeating),
\end{itemize}
then the descending chain
$
C_{\le 1}M \supseteq C_{\le 2}M \supseteq C_{\le 3}M \supseteq \dots
$
stabilizes.
\end{lem}

\begin{proof}
Case (i) is clear. For case (ii), suppose that $C$ is direct.
If $P$ is a primitive cycle with the same head as $C$ and length $r$,
then the first letter $C_1$ cannot be the same as $P_1$,
for that would force $C = P^\infty$, which is direct and repeating.
Thus $P_r C_1 = 0$ in $\Lambda$, so $P C_1 M = 0$. It follows that $C_1 M\subseteq Z$,
where $Z = \{ m\in e_v M : zm=0\}$ and $v$ is the head of $C$.
The hypothesis on $M$ ensures that $Z$ is finite dimensional, 
so the terms in the descending chain are finite-dimensional, so it must stabilize.

If $C$ is eventually inverse the chain stabilizes at $C_{\le n} M$
with $n$ chosen so that $C_{>n}$ is inverse.

Thus we may suppose $C$ is not direct and not eventually inverse.
It follows that $C= D x^{-1}y B$ for some words $D,B$,
and distinct arrows $x,y$, say with head $v$.
We need the chain $D x^{-1} y B_{\le n}M$ to stabilize. 
This holds since $Dx^{-1} y B_{\le n}M = Dx^{-1}(xM \cap y B_{\le n}M)$,
and $xM \cap y B_{\le n}M$ is finite dimensional by Lemma~\ref{l:fdintersection},
so the chain $xM \cap y B_{\le n}M$ stabilizes.
\end{proof}

\begin{lem}[Realization lemma]
\label{l:realize}
If $M$ is finitely controlled or pointwise artinian and $C$ is an $\N$-word, then $C^+(M) = \bigcap_{n\ge 0} C_{\le n} M$.
\end{lem}

\begin{proof}
It suffices to show that if $\ell D$ is a word with $\ell$ a letter, then
\[
\bigcap_{n\ge 0} \ell D_{\le n} M \subseteq \ell \left( \bigcap_{n\ge 0} D_{\le n} M\right).
\]
This is trivial if $\ell$ is an inverse letter, so suppose $\ell$ is a direct letter.
If the descending chain $D_{\le n}M$ stabilizes, the result is clear.
Thus by Lemma~\ref{l:cplustrunc} we may suppose
that $D$ is direct and repeating.
Then, since $\ell$ is a direct letter, so is $\ell D$.
Thus $\ell D = P^\infty$ for a primitive cycle $P = \ell B$.
Then 
\[
\bigcap_{n\ge 0} \ell D_{\le n} M = 
\bigcap_{m\ge 0} P^m M
\]
and by Lemma~\ref{l:ptorsion} this is contained in
\[
P \left( \bigcap_{m\ge 0} P^m M \right)
\subseteq \ell \left( \bigcap_{m\ge 0} BP^m M \right) = \ell \left( \bigcap_{n\ge 0} D_{\le n} M \right).
\]
\end{proof}

\begin{lem}[Weak covering property]
\label{l:covering}
Let $M$ be a $\Lambda$-module,
let $v$ be a vertex and $\epsilon=\pm1$.
Suppose that $S$ is a non-empty subset of $e_v M$ with $0\notin S$.
Then there is a word $C\in \W_{v,\epsilon}$ such that
either (a) $C$ is finite and
$S$ meets $C^+(M)$ but does not meet $C^-(M)$,
or (b) $C$ is an $\N$-word and
$S$ meets $C_{\le n}M$ for all $n$ but does not meet $C^-(M)$.
\end{lem}

\begin{proof}
Suppose there is no finite word $C\in \W_{v,\epsilon}$ such that $S$ meets $C^+(M)$ but
not $C^-(M)$.
Starting with the trivial word $1_{v,\epsilon}$,
we iteratively construct an $\N$-word $C\in \W_{v,\epsilon}$
such that $S$ meets $C_{\le n}M$ but not $C_{\le n}0$.
Suppose we have constructed $D = C_{\le n}$.
If there is a letter $y$ with $D y$ a word, and $S$ meets $D y M$, then
we define $C_{n+1} = y$ and repeat. Otherwise $S$ does not meet $D^-(M)$.
If there is a letter $x$ with $D x^{-1}$ a word, and $S$ does not meet $D x^{-1} 0$
then we define $C_{n+1} =  x^{-1}$ and repeat. Otherwise $S$ meets $D^+(M)$.
By our assumption, one of these two possibilities must occur.
\end{proof}

\begin{lem}[Covering property for one-sided functors]
Let $M$ be a $\Lambda$-module,
let $v$ be a vertex and $\epsilon=\pm1$.
Suppose $U$ is a $k[z]$-submodule of $e_v M$, $H$ is a subset of $e_v M$ and $m\in H\setminus U$.
Suppose that either $M$ is pointwise artinian, or that $M$ is finitely controlled and $z e_v M \subseteq U$.
Then there is a word $C\in \W_{v,\epsilon}$ such that $H\cap (U+m)$ meets $C^+(M)$ but does not meet $C^-(M)$.
\end{lem}

\begin{proof}
The set $S = H \cap (U+m)$ contains $m$ but not 0, so the weak covering property gives a word $C$ such
that $S$ does not meet $C^-(M)$. If $C$ is a finite word, then $S$ meets $C^+(M)$, as required.
If $C$ is an $\N$-word, and $C$ is not direct and repeating, then by Lemma \ref{l:cplustrunc}
and the realization lemma, $S$ doesn't meet $C^+(M)$, as required.
Thus suppose $C$ is direct and repeating. Then $C = P^\infty$ for some primitive cycle $P$.
Then $U+m$ meets $P^2 M = z P M \subseteq U$, contradicting that $m\notin U$.
\end{proof}

\begin{lem}[Covering property for refined functors]
\label{l:refinedcovering}
Let $M$ be a $\Lambda$-module, and let $v$ be a vertex.
Suppose $U$ is a $k[z]$-submodule of $e_v M$ and $m\in e_v M\setminus U$.
Suppose that either $M$ is pointwise artinian, or that $M$ is finitely controlled and $z e_v M \subseteq U$.
Then $U+m$ meets $G^+_{B,D}(M)$ but not $G^-_{B,D}(M)$
for some $(B,D)$.
\end{lem}

\begin{proof}
By the covering property for one-sided functors, with $H = e_v M$ there is $B$ with head $v$ and of the correct sign such
that $U+m$ meets $B^+(M)$ but not $B^-(M)$.
Then we can write $U+m = U+m'$ for some $m'\in B^+(M)$.
Letting $U' = U+B^-(M)$ we have $m'\notin U'$.
We now apply the covering property for one-sided functors with the submodule $U'$ and $H = B^+(M)$ and the element $m'$,
to get a word $D$ with head $v$ and of the right sign, such that $B^+(M) \cap (U'+m')$ meets $D^+(M)$ but not $D^-(M)$.
It follows that $U+m$ meets $G^+_{B,D}(M)$ but not $G^-_{B,D}(M)$.
\end{proof}

\begin{lem}
\label{l:cokprop}
Suppose $\theta:N\to M$ is a homomorphism such that $F_{B,D}(\theta)$ is an isomorphism for all refined functors $F_{B,D}$. 
\begin{itemize}
\item[(i)]
If $M$ is pointwise artinian, then $\theta$ is surjective.
\item[(ii)]
If $M$ is finitely controlled, then the cokernel of $\theta$ is primitive torsion.
\end{itemize}
\end{lem}

\begin{proof}
In case (i), if $\theta$ is not surjective, say $e_v \Ima(\theta) \neq e_v M$, let $U = e_v \Ima(\theta)$
and choose $m\in e_v M\setminus U$.
In case (ii), if the cokernel of $\theta$ is not primitive torsion, choose a vertex $v$ with $e_v M/ e_v \Ima(\theta)$ not primitive torsion.
Then this module has a 1-dimensional quotient killed by $z$, so there is a $k[z]$-submodule $U$ of codimension 1 in $e_v M$ with
$e_v \Ima(\theta)\subseteq U$ and $z e_v M \subseteq U$. Choose $m\in e_v M \setminus U$.

The covering property for refined functors gives $B,D$ such that $U+m$ meets $G_{B,D}^+(M)$ but not $G_{B,D}^-(M)$.
Thus there are $u\in U$, $b\in B^-(M)$ and $d\in B^+(M)\cap D^+(M)$ such that $u+m = b+d$.
Since $\theta$ induces an isomorphism in refined functors, there is $n\in e_v N$ with $d = \theta(n) + c + c'$ with $c\in D^-(M)\cap B^+(M)$
and $c' \in D^+(M)\cap B^-(M)$.
Then $\theta(n)\in U$, so $U+m$ contains $b+c+c'$, so it meets $G_{B,D}^-(M)$, a contradiction.
\end{proof}

\begin{proof}[Proof of Theorem~\ref{t:mainartinian}]
The map $\theta:N\to M$ of Theorem~\ref{t:existdecomp} is 
injective by Lemma~\ref{l:sbmono} and surjective by Lemma~\ref{l:cokprop}, so an isomorphism.
\end{proof}

\section{Extensions by a primitive simple}
\label{s:primsimpleexts}
We fix a \emph{primitive simple} $S$ for $\Lambda$, that is, a simple, primitive torsion module.
It is easy to see (for example using Theorem~\ref{t:mainartinian})
that it is of the form $S = M({}^\infty P^\infty,V)$
where $P$ is a primitive cycle, say with head $v$, sign $\epsilon$ and length $p$,
and $V$ is a simple $k[T,T^{-1}]$-module,
so of the form $V = k[T,T^{-1}]/(f(T))$ where $f(T)$ is an irreducible polynomial in $k[T]$ with $f(0)=1$.
Since $P$ has sign $\epsilon$, it follows that $P^{-1}$ and $(P^{-1})^\infty$ have sign $-\epsilon$.

\begin{dfn}
Let $C$ be an $I$-word.
We say that $i\in I$ is \emph{$P$-deep for $C$} if $C(i,-\epsilon) = (P^{-1})^\infty$.
Equivalently if the basis element $b_i$ in $M(C)$ is not killed by any power of $P$.
We say that $i\in I$ is a \emph{$P$-peak for $C$} if it is $P$-deep for $C$ and $C(i,\epsilon)$ is not of the form $PD$ for some word $D$.
Equivalently, it is $P$-deep for $C$ and $b_i$ is not in $P M(C)$.
\end{dfn}

Clearly only an infinite word can have a $P$-peak, and then it has at most two $P$-peaks (and if so it is a $\Z$-word).
Our aim in this section is to prove the following result.

\begin{thm}[Extension Theorem]
\label{t:extension}
Suppose that $M$ is a finitely controlled $\Lambda$-module
and $N$ is a submodule of $M$ with $\tau^1(M) \subseteq N$
and $M/N\cong S$.
Suppose that $N$ is a direct sum of string and finite-dimensional band modules,
\[
N = \bigoplus_{\lambda\in \Phi} N_\lambda,
\]
indexed by some set $\Phi$.
Then there is some $\mu\in\Phi$ with
$N_\mu$ of the form $M(C)$ for some word $C$,
which has a $P$-peak,
such that $M = N'_\mu \oplus N'$ where
\[
N' = \bigoplus_{\lambda\in \Phi\smallsetminus\{\mu\}} N_\lambda,
\]
and $N_\mu'$ is a submodule of $M$ with $N'_\mu \cong N_\mu$.
\end{thm}

The following is straightforward.

\begin{lem}
\label{l:projres}
There is a projective resolution
\[
0 \to \Lambda e_v \to \Lambda e_v \to S\to 0
\]
where the first map is right multiplication by $f(P)$.
\end{lem}

For any $\Lambda$-module $M$, the resolution of $S$ gives an exact sequence
\[
0\longrightarrow \Hom(S,M)\longrightarrow e_v M\xrightarrow{f(P)} e_v M \xrightarrow{\alpha_M} \Ext^1(S,M)\longrightarrow 0.
\]
We denote the pullback of $\xi\in\Ext^1(S,M)$ along $a\in\End(S)$ by $\xi a$,
and if $\theta:M\to N$ is a homomorphism, we denote the pushout map $\Ext^1(S,M)\to \Ext^1(S,N)$ by $\theta_*$.

\begin{lem}
\label{l:extswap}
If $a\in\End(S)$ and $\xi\in\Ext^1(S,M)$, then $\xi a = \psi_*(\xi)$ for some $\psi$
in the centre of $\End(M)$.
\end{lem}

\begin{proof}
For any $\Lambda$-module $M$, the action of $k[z]$ on $M$ defines a homomorphism
$\gamma_M:k[z]\to \End(M)$.
If $N$ is another $\Lambda$-module, the actions of $k[z]$ on $M$ and $N$ induce
left and right actions of $k[z]$ on $\Hom(N,M)$, but these are the same since
the action of $z$ on $e_v M$ or $e_v N$ is given by multiplication by $z_v\in \Lambda$.
Using a projective resolution of $N$, the same
holds for the two actions of $k[z]$ on $\Ext^1(N,M)$.
It is clear that $\gamma_S$ induces an isomorphism
\[
k[z]/(f(z)) \cong \End(S).
\]
Thus, writing
$a = \gamma_S(h(z))$ for some $h(z)\in k[z]$, we can take $\psi = \gamma_M(h(z))$.
It is central by the discussion above.
\end{proof}

If $C$ is an $I$-word and $i$ is a $P$-peak for $C$, consider the exact sequence
\[
\xi_{C,i} : 0 \to M(C) \to E_{C,i} \to S\to 0
\]
formed from the pushout of the projective resolution in Lemma~\ref{l:projres}
along the homomorphism $\Lambda e_v\to M(C)$ sending $e_v$ to $b_i$.
Thus
\[
\xi_{C,i} = \alpha_{M(C)}(b_i) \in \Ext^1(S,M(C)).
\]

\begin{lem}
The middle term $E_{C,i}$ of the exact sequence $\xi_{C,i}$ is isomorphic to $M(C)$.
\end{lem}

\begin{proof}
We define $\phi\in \End(M(C))$
as follows.
If $d(C,-\epsilon)=1$, so that $C_{>i} = (P^{-1})^\infty$,
let $j$ be minimal with $C_{>j}$ an inverse word.
Since $i$ is a $P$-peak for $C$, we have $i-p < j \le i$, where $p$ is the length of $P$.
We define $\phi(b_k) = b_{k+p}$ for $k\ge j$ and $\phi(b_k)=0$ for $k<j$.
Dually, if $d_i(C,-\epsilon) = -1$, so that $(C_{\le i})^{-1} = (P^{-1})^\infty$,
let $j\in I$ be maximal such that $(C_{\le j})^{-1}$ is an inverse word,
and define $\phi(b_k) = b_{k-p}$ for $k\le j$
and $\phi(b_k)=0$ for $k>j$.

It is straightforward to see that
$f(\phi)$ is an injective endomorphism of $M(C)$ with cokernel isomorphic to $S$.
We fix an isomorphism between $S$ and the cokernel of $f(\phi)$, and hence obtain an exact sequence
\[
\eta_{C,i} : 0 \to M(C) \xrightarrow{f(\phi)} M(C) \xrightarrow{g} S\to 0.
\]
Let $M=M(C)$. The exact sequences above lead to a commutative diagram with exact rows and columns

\[
\begin{CD}
& & 0 & & 0 & & 0
\\
& & @VVV @VVV @VVV
\\
0 @>>> \Hom(S,M) @>>> \Hom(S,M) @>>> \End(S)
\\
& & @VVV @VVV @VVV
\\
0 @>>> e_v M @>f(\phi)>> e_v M @>g>> e_v S @>>> 0
\\
& & @Vf(P)VV @Vf(P)VV @Vf(P)VV
\\
0 @>>> e_v M @>f(\phi)>> e_v M @>g>> e_v S @>>> 0
\\
& & @V\alpha_M VV @V\alpha_M VV @V\alpha_S VV
\\
& & \Ext^1(S,M) @>>> \Ext^1(S,M) @>>> \Ext^1(S,S)
\\
& & @VVV @VVV @VVV
\\
& & 0 & & 0 & & 0.
\end{CD}
\]

\noindent
The Snake Lemma gives a connecting map $c:\End(S)\to \Ext^1(S,M)$ sending $a\in\End(S)$ to $\eta_{C,i}a$.
Now $f(\phi) b_i = f(P) b_i$ so by the diagram chase defining the connecting map
there is $a\in\End(S)$ with $c(a) = \alpha_M(b_i)$.
Moreover $a\neq 0$ since $b_i\notin f(\phi) M$, so $g(b_i)\neq 0$.
Then $\eta_{C,i}a = \alpha_M(b_i) = \xi_{C,i}$, so
there is map of exact sequences
\[
\begin{CD}
\xi_{C,i}: &0 @>>> M(C) @>>> E_{C,i} @>>> S  @>>> 0
\\
&& & @| @VVV @V a VV
\\
\eta_{C,i}: &0 @>>> M(C) @>>> M(C) @>>> S  @>>> 0
\end{CD}
\]
and since $a$ is an isomorphism, $E_{C,i}\cong M(C)$.
\end{proof}

\begin{lem}
\label{l:stringextbasis}
If $C$ is a word which is not equivalent to ${}^\infty P^\infty$,
then the elements $\xi_{C,i}$ with $i$ a $P$-peak for $C$,
form an $\End(S)$-basis for $\Ext^1(S,M(C))$.
\end{lem}

\begin{proof}
Observe that $e_v M(C)$, as a $k[P]$-module, is the direct sum of free submodules $k[P]b_i$
where $i$ runs through the $P$-peaks, and a nilpotent torsion submodule spanned
by the $b_i$ with $v_i(C)=v$ and $i$ not $P$-deep.
Now the isomorphism $e_v M(C) / f(P) e_v M(C) \to \Ext^1(S,M(C))$ induced by $\alpha_{M(C)}$
gives the result.
\end{proof}

Let $\Sigma$ be a set of representative of the equivalence classes of words.

\begin{dfn}
We define a \emph{$P$-class} to be a pair $(C,i)$
where $C\in\Sigma$ and $i$ is a $P$-peak for $C$,
The set of $P$-classes is totally ordered by
$(C,i) > (D,j)$ if $C(i,\epsilon) > D(j,\epsilon)$.
\end{dfn}

Henceforth, we write $b_i^C$ instead of $b_i$ for the basis elements of $M(C)$, so as to identify the word $C$.

\begin{lem}
\label{l:peakhom}
Suppose that $(C,i) > (D,j)$ are $P$-classes.
Then there is a homomorphism $\theta_{ij} : M(C)\to M(D)$
such $(\theta_{ij})_*(\xi_{C,i}) = \xi_{D,j}$.
Moreover, if $C=D$ then $\theta_{ij}^2 = 0$.
\end{lem}

\begin{proof}
By assumption $C(i,\epsilon) > D(j,\epsilon)$.
Let $r$ be maximal with
\[
C(i,\epsilon)_{\le r} = D(j,\epsilon)_{\le r} = B,
\]
say. Then $C(i,\epsilon)_{r+1}$ is an inverse letter and $D(j,\epsilon)_{r+1}$ is a direct letter
(or one of them is absent if the relevant word $C(i,\epsilon)$ or $D(j,\epsilon)$ has length~$r$).
Let $c = d_i(C,\epsilon)$ and $d = d_j(D,\epsilon)$.
We define
\[
\theta_{ij}(b_k^C) =
\begin{cases}
b_{j-c d (i-k)}^D & (c(i-k)\ge -r)
\\
0 & (c(i-k) < -r).
\end{cases}
\]
Then $\theta_{ij}$ is is a homomorphism from $M(C)$ to $M(D)$ sending $b_i^C$ to $b_j^D$.
Thus
\[
(\theta_{ij})_*(\xi_{C,i}) = (\theta_{ij})_*(\alpha_{M(C)}(b_i^C)) = \alpha_{M(D)}(b_j^D) = \xi_{D,j}.
\]
Now suppose that $C=D$.
Then $C(i,\epsilon)$ is of the form $E (P^{-1})^{\infty}$
and $C(j,\epsilon)$ is of the form $E^{-1} (P^{-1})^{\infty}$,
where $E$ has length $\lvert i-j\rvert$.
Then $E > E^{-1}$ and $r$ is maximal with $E_{\le r} = (E^{-1})_{\le r}$.
Then Lemma~\ref{l:symmetry} implies that $E$ has length $> 2r$, and that $E = B F B^{-1}$
for some word $F$ of length $\ge 1$ whose first and last letters are inverse.
But then the basis elements $b_k^C$ in the image of $\theta_{ij}$ are all sent to zero by~$\theta_{ij}$.
\end{proof}

\begin{lem}
\label{l:extendband}
Let $M = M(D,U)$ be a finite-dimensional band module.
\begin{itemize}
\item[(i)]
If $D$ is not equivalent to ${}^\infty P^\infty$ then $\Ext^1(S,M) = 0$.
\item[(ii)]
If $D$ is equivalent to ${}^\infty P^\infty$ then $\Ext^1(S,M)$ has dimension $\le 1$ as a
vector space over $\End(S)$.
\item[(iii)]
If $\Ext^1(S,M)\neq 0$ and $(C,i)$ is a $P$-class,
then $\psi_*(\xi_{C,i})\neq 0$ for some homomorphism $\psi:M(C)\to M$.
\end{itemize}
\end{lem}

\begin{proof}
(i) The projective resolution of $S$ realizes $\Ext^1(S,M)$ as the cokernel of
the map $f(P)$ from $e_v M$ to $e_v M$.
If $D$ is not equivalent to ${}^\infty P^\infty$ then there are no $P$-deep basis
elements for $D$. It follows that each element of $e_v M$ is killed by a power of $P$,
so $f(P)$ acts invertibly on $e_v M$.

(ii) We may assume that $D = {}^\infty P^\infty$.
We have $M = U_0 \oplus U_1 \oplus \dots \oplus U_{p-1}$ using the notation
preceding Lemma~\ref{l:bandeval}, where $p$ is the length of $P$.
Now as a $k[P]$-module, $e_v M$ is isomorphic to the direct sum of
$U_0$, which is a copy of $U$
with $P$ acting as $T$, and a nilpotent torsion submodule, spanned by
the other $U_i$ with $U_i = e_v U_i$.
Thus
\[
\Ext^1(S,M) \cong e_v M/f(P)M \cong U/f(T)U \cong \Ext^1(V,U).
\]
Since $U$ is an indecomposable $k[T,T^{-1}]$-module and $V$ is simple,
this has dimension $\le 1$ as a module for $\End(V)\cong \End(S)$.

(iii) We may assume we are in case (ii).
Then $\Ext^1(V,U)\neq 0$,
so we can identify $U = k[T]/(f(T))^r$ for some $r>0$.
There is a homomorphism $M(C)\to M(D)$ sending $b_i^C$ to $b_0^D$.
It induces a homomorphism $\psi:M(C)\to M$ sending
$b_i^C$ to $m = b_0^D \otimes \overline 1 \in e_v M$, and
\[
\psi_*(\xi_{C,i}) = \psi_*(\alpha_{M(C)}(b_i^C)) = \alpha_M(\psi(b_i^C)) = \alpha_M(m).
\]
This is non-zero since $m\notin f(P)M$,
which follows from the observation in (ii)
about the $k[P]$-module structure of $e_v M$,
as we can identify $m$ with the element $\overline 1\in U_0$.
\end{proof}

\begin{proof}[Proof of Theorem \ref{t:extension}]
Letting $i_\lambda$ denote the inclusion of $N_\lambda$ in $N$,
we can write the class $\zeta\in \Ext^1(S,N)$ of the extension
\[
0\to N\to M\to S\to 0
\]
as
\[
\zeta = \sum_{\lambda\in \Phi} (i_\lambda)_*(\zeta_\lambda)
\]
for elements $\zeta_\lambda\in \Ext^1(S,N_\lambda)$, all but finitely many zero.

If $N_\lambda$ is a string module, since equivalent words give isomorphic string modules,
we may assume that it is of the form $M(C^\lambda)$
with $C^\lambda\in \Sigma$, our chosen set of representative of the equivalence classes of words,
and by Lemma~\ref{l:stringextbasis} we can write
\[
\zeta_\lambda = \sum_i \xi_{C^\lambda,i} a_{\lambda i}
\]
where $i$ runs through the $P$-peaks for $C^\lambda$ and $a_{\lambda i}\in\End(S)$.

There must be at least one string module $N_\lambda$ with $\zeta_\lambda\neq 0$,
for otherwise, by Lemma~\ref{l:extendband}, $S$ only extends band
modules which are primitive torsion, so there is a primitive torsion submodule
of $M$ mapping onto $S$, contradicting the assumption that $\tau^1(M)\subseteq N$.
Among all pairs $(\lambda,i)$ where $N_\lambda$ is a string module $M(C^\lambda)$,
$i$ is a $P$-peak for $C^\lambda$ and $a_{\lambda i}\neq 0$,
choose a pair $(\lambda,i)$ for which the $P$-class $(C^\lambda,i)$ is maximal. We denote it $(\mu,j)$.

Suppose that $N_\lambda$ is a band module and $\zeta_\lambda \neq 0$.
By Lemma~\ref{l:extendband}
there is a map $\theta_\lambda : N_\mu\to N_\lambda$ such that
$(\theta_\lambda)_*(\xi_{C^\mu,j})\neq 0$.
Then by Lemma~\ref{l:extswap} and Lemma~\ref{l:extendband}(ii)
there is $\psi_\lambda\in\End(N_\mu)$ such that
$\phi_\lambda = \psi_\lambda \theta_\lambda$ satisfies
$(\phi_\lambda)_*(\xi_{C^\mu,j} a_{\mu j}) = \zeta_\lambda$.

Suppose that $N_\lambda$ is a string module and $\zeta_\lambda\neq 0$.
If $i$ is a $P$-peak for $C^\lambda$ with
$(\lambda,i) \neq (\mu,j)$ and $a_{\lambda i}\neq 0$, then by
the choice of $(\mu,j)$, by Lemma~\ref{l:peakhom}
(or trivially if $(C^\lambda,i)=(C^\mu,j)$),
there is a homomorphism $\theta_{\lambda i} : N_\mu\to N_\lambda$
such that $(\theta_{\lambda i})_*(\xi_{C^\mu,j}) = \xi_{C^\lambda,i}$.
By Lemma \ref{l:extswap} there is $\psi_{\lambda i}$ in the centre of $\End(N_\lambda)$
such that $(\psi_{\lambda i}\theta_{\lambda i})_*(\xi_{C^\mu,j} a_{\mu j}) =\xi_{C^\lambda,i} a_{\lambda i}$.
We define $\phi_\lambda:N_\mu\to N_\lambda$ by
\[
\phi_\lambda =
\begin{cases}
\sum_i \psi_{\lambda i}\theta_{\lambda i}
&
\text{(if $\lambda\neq\mu$)}
\\
1 + \sum_i \psi_{\lambda i}\theta_{\lambda i}
&
\text{(if $\lambda =\mu$)}
\end{cases}
\]
where $i$ runs through the $P$-peaks for $C^\lambda$ (with $i\neq j$ in case $\lambda=\mu$,
so the second sum has at most one term).
It follows that
$(\phi_\lambda)_*(\xi_{C^\mu,j} a_{\mu j}) = \zeta_\lambda$.
Observe that $\phi_\mu$ is invertible
since $\psi_{\mu i}$ is in the centre of $\End(N_\mu)$, so
$
(\psi_{\mu i}\theta_{\mu i})^2 = \psi_{\mu i}^2 \theta_{\mu i}^2 = 0
$
by Lemma \ref{l:peakhom}.

Now consider the pullback diagram
\[
\begin{CD}
\xi_{C^\mu,j}a_{\mu j}: &0 @>>> N_\mu @>p>> E @>q>> S  @>>> 0
\\
&& & @| @VrVV @Va_{\mu j}VV
\\
\xi_{C^\mu,j}: &0 @>>> N_\mu @>>> N_\mu @>>> S  @>>> 0.
\end{CD}
\]
Since $a_{\mu j}$ is an isomorphism, so is $r$.
The map $\phi = \sum_\lambda i_\lambda \phi_\lambda:N_\mu\to N$
satisfies $\phi_*(\xi_{C^\mu,j}a_{\mu j}) = \zeta$, so there is a
pushout diagram
\[
\begin{CD}
\xi_{C^\mu,j}a_{\mu j}: &0 @>>> N_\mu @>p>> E @>q>> S  @>>> 0
\\
&& & @V \phi VV @V t VV @|
\\
\zeta: &0 @>>> N @>>> M @>>> S  @>>> 0.
\end{CD}
\]
Since $\phi_\mu$ is invertible,
$\phi$ is a split monomorphism and $N = N'\oplus\Ima(\phi)$.
It follows that $M = N'\oplus \Ima(t)$
and $\Ima(t)\cong E\cong N_\mu$.
\end{proof}

\section{Proofs of the main results}
\label{s:main}
Theorem~\ref{t:mainartinian} has already been proved in \S\ref{s:cover}.

\begin{proof}
[Proof of Theorem~\ref{t:props}]
It is known that string modules are indecomposable:
see Krause~\cite{K} for a special case and \cite[\S 1.4]{CBinfdim} in general.
If $M(C,V)$ is a finite-dimensional or primitive injective band module, then it is artinian, so if it were to decompose, 
each of the summands would be a direct sum of string and band modules by Theorem~\ref{t:mainartinian}.
But then Theorem~\ref{t:summands} ensures that string module summands and other bands do not occur, and
gives a decomposition of $V$. But since $M(C,V)$ is a band module, $V$ is indecomposable.
The statement about isomorphisms follows from Theorem~\ref{t:summands}.
\end{proof}

\begin{proof}[Proof of Theorem \ref{t:main}]
We may suppose that $Q$ is connected.
Theorem~\ref{t:existdecomp} and Lemma~\ref{l:sbmono} 
give a submodule $N$ of $M$, such that
\[
N = \bigoplus_{\lambda\in\Phi} N_\lambda,
\]
a direct sum of string and finite-dimensional band modules. Moreover
$N$ contains $\tau^1(M)$ by Lemma~\ref{l:contprimtor},
and $L=M/N$ is primitive torsion by Lemma~\ref{l:cokprop}.

Since $Q$ is connected it has only countably many vertices,
and since $L$ is finitely controlled and primitive torsion,
$e_v L$ is finite-dimensional for all $v$. It follows that
we can write $L$ as a union $L = \bigcup L_j$
of a finite or infinite sequence of submodules
\[
0=L_0 \subset L_1 \subset L_2 \subset \dots
\]
with the quotients $S_j = L_j/L_{j-1}$ being primitive simples.
Let $M_j$ be the inverse image of $L_j$ in $M$. Thus we have exact sequences
\[
0\to M_{j-1}\to M_j\to S_j\to 0
\]
with $M_0 = N$ and $M = \bigcup_n M_n$.

Let $N_{\lambda,0} = N_\lambda$.
By Theorem~\ref{t:extension} we can write $M_j = \bigoplus_{\lambda\in\Phi} N_{\lambda,j}$
for submodules $N_{\lambda,j} \cong N_\lambda$ and such that
$N_{\lambda,j} = N_{\lambda,j-1}$ unless $N_\lambda$ is isomorphic a string module $M(C)$
such that $C$ has a $P$-peak for some primitive cycle $P$ with $S_j$ supported at the head of $P$.

For any vertex $v$, only finitely many of the simples $S_j$ can be supported at $v$.
It follows that for each $\lambda$ there is some $j$ with
\[
N_{\lambda,j} = N_{\lambda,j+1} = N_{\lambda,j+2} = \dots.
\]
Defining $N_{\lambda,\infty} = N_{\lambda,j}$, it follows easily that $M = \bigoplus_{\lambda\in\Phi} N_{\lambda,\infty}$.
\end{proof}

\begin{proof}[Proof of Theorem~\ref{t:KRS}]
By Theorems~\ref{t:main} and~\ref{t:mainartinian}
the indecomposable summands are string and band modules.
The result thus follows from Theorem~\ref{t:summands} and the Krull-Remak-Schmidt
property for finite-dimensional or artinian $k[T,T^{-1}]$-modules.
\end{proof}

Finally, from Lemmas~\ref{l:kzfc} and~\ref{l:kzpa}, one easily obtains the following
characterization of direct sums of string and band modules which are finitely controlled or pointwise artinian.

\begin{pro}
If $M$ is a direct sum of string and finite-dimensional band modules, then $M$ is 
\begin{itemize}
\item[(i)]
finitely generated if and only if, for any string module $M(C)$ which occurs, $C$ and $C^{-1}$ are eventually inverse,
and the sum is finite;
\item[(ii)]
finitely controlled if and only if, for any string module $M(C)$ which occurs, $C$  and $C^{-1}$ are eventually inverse or right vertex-finite,
and, for every vertex $v$, only finitely many summands are supported at $v$.
\end{itemize}
\end{pro}

\begin{pro}
If $M$ is direct sum of string modules, finite-dimensional band modules and primitive injective band modules, then $M$ is 
\begin{itemize}
\item[(iii)]
artinian if and only if, for any string module $M(C)$ which occurs, $C$ and $C^{-1}$ are eventually direct,
and the sum is finite;
\item[(iv)]
pointwise artinian if and only if, for any string module $M(C)$ which occurs, $C$  and $C^{-1}$ are eventually direct or right vertex-finite,
and, for every vertex $v$, only finitely many summands are supported at $v$. 
\end{itemize}
\end{pro}

\section*{Acknowledgement}

The author is grateful to C.~Ricke for pointing out a gap in a proof in the first draft of this article posted to the Arxiv.
It is corrected using Lemma~\ref{l:sbmono}.

\end{document}